\documentclass[12pt]{amsart}

\usepackage{amssymb}
\usepackage{mathrsfs}
\usepackage{color}
\usepackage{enumitem}

\usepackage{anysize}
\marginsize{2cm}{2cm}{2cm}{2cm}

\renewcommand{\le}{\leqslant}
\renewcommand{\ge}{\geqslant}

\usepackage{versions}
\usepackage{hyperref}
\hypersetup{
  colorlinks   = true, 
  urlcolor     = blue, 
  linkcolor    = blue, 
  citecolor   = red 
}
\usepackage{cancel}
\usepackage{yhmath}

\newtheorem{theorem}{Theorem}[section]

\newtheorem{lemma}[theorem]{Lemma}

\newtheorem{corollary}[theorem]{Corollary}

\theoremstyle{definition}
\newtheorem{definition}[theorem]{Definition}

\theoremstyle{remark}

\newtheorem{example}[theorem]{Example}

\numberwithin{equation}{section}

{%
  \goodbreak
}
{%
  \goodbreak
}
{%
  ~]]\goodbreak\smallskip
}
\excludeversion{solution}
\excludeversion{answer}

\DeclareMathOperator{\ind}{ind}
\DeclareMathOperator{\coind}{co-ind}
\DeclareMathOperator{\hind}{h-ind}

\newcommand{\id}{\mathop{\rm id}\nolimits}

\newcommand{\where}{\mathop{\ |\ }}

\newcommand{\hchi}{\mathop{\operatorname{h-\chi}}}

\renewcommand{\epsilon}{\varepsilon}
\renewcommand{\phi}{\varphi}
\renewcommand{\kappa}{\varkappa}
\renewcommand{\theta}{\vartheta}

\title{Hedetniemi's conjecture from the topological viewpoint}

\author{Hamid Reza Daneshpajouh{$^\spadesuit$}}
\author{Roman~Karasev{$^\clubsuit$}}
\author{Alexey~Volovikov{$^\diamondsuit$}}

\thanks{{$^\clubsuit$} Supported by the Federal professorship program grant 1.456.2016/1.4 and the Russian Foundation for Basic Research grants 18-01-00036 and 19-01-00169}

\address{Hamid Reza Daneshpajouh, School of Mathematics, Institute for Research in Fundamental Sciences (IPM), P.O. Box 19395-5746, Tehran, Iran.}
\email{hr.daneshpajouh@ipm.ir}

\address{Roman~Karasev, Moscow Institute of Physics and Technology, Institutskiy per. 9, Dolgoprudny, Russia 141700\newline \indent
Institute for Information Transmission Problems RAS, Bolshoy Karetny per. 19, Moscow, Russia 127994}
\email{r\_n\_karasev@mail.ru}
\urladdr{http://www.rkarasev.ru/en/}

\address{Alexey Volovikov, Department of Higher Mathematics, MTU MIREA, Vernadskogo av., 78, Moscow, 119454, Russia}
\email{a\_volov@list.ru}


\begin{document}

\begin{abstract}
This paper is devoted to studying a topological version of the famous Hedetniemi conjecture which says:
The $\mathbb Z/2$-index of the Cartesian product of two $\mathbb Z/2$-spaces is equal to the minimum of their $\mathbb Z/2$-indexes.
We fully confirm the version of this conjecture for the homological index via establishing  a stronger formula for the homological index of the join of $\mathbb Z/2$-spaces. Moreover, we confirm the original conjecture for the case when one of the factors is an $n$-sphere. 
Analogous results for $\mathbb Z/p$-spaces are presented as well. In addition, we answer a question about computing the index of some non-trivial products, raised by Marcin Wrochna. Finally, some new topological lower bounds for the chromatic number of the Categorical product of (hyper-)graphs are presented. 
\end{abstract}

\maketitle

\section{Introduction}

Hedetniemi conjectured \cite{hedetniemi1966} that the chromatic number of the categorical product of two graphs equals the minimum of their chromatic numbers. This conjecture has been open for more than 50 years, and this year a counterexample for this conjecture was found by Shitov \cite{shitov2019}.

Although Hedetniemi's conjecture has been disproved, we believe that it makes sense to study the techniques around it and continue establishing its validity under additional assumptions. 

As for positive or conditionally positive results around Hedetniemi's conjecture, we must mention the connection \cite{wrochna2019} between this conjecture and the conjecture about the mapping index of polyhedra with free involution, $\ind X\times Y = \min\{\ind X, \ind Y\}$, see the textbook \cite{matousek2003using} or Section \ref{section:mapping-index} below for the definition of the mapping index.

Shitov's result shows that computing the chromatic number of the categorical product of graphs is not an easy question, the same, in our opinion, applies to the mapping index of the product conjecture. We believe so because the mapping index is intrinsically connected to the notion of higher topological obstructions and higher homotopy groups of spheres, the topics that after decades of study mostly remain mysterious and extremely hard to handle. In view of this we somehow switch to ``changing the definition'' and show ``the right form'' of these results, replacing the mapping index by the homological index \cite{yang1954,yang1955,matousek2003using}. In particular, for the homological index of polyhedra with free involutions we are able to prove 
\begin{equation}
\label{equation:hind-prod}
\hind X\times Y = \min\{\hind X, \hind Y\},
\end{equation}
see Section \ref{section:h-ind-2} below for the details. We should emphasize that we have found three different ways to prove \eqref{equation:hind-prod}, each of them demonstrating a different technique. We present one method in this paper, the one which we believe is shorter and proving other useful formulas along the way. For the other proofs of \eqref{equation:hind-prod}, see \cite{daneshpajouh2018hedetniemi}.

Formula \eqref{equation:hind-prod} automatically implies the validity of Hedetniemi's conjecture for graphs, whose homological lower bound for the chromatic number, the homological index of the box complex plus $2$, is tight and equals the actual chromatic number. Note that since the founding works of Lov\'asz \cite{lovasz1978} (see also the textbook \cite{matousek2003using}) the standard way to establish lower bounds on the chromatic number of a graph $H$ was through establishing lower bounds for the mapping index of its so-called box complex $\mathcal B(H)$. Moreover, there is the inequality between the homological index and the mapping index, $\hind X \le \ind X$, and in all practical cases what was bounded from below was the homological index $\hind \mathcal B(H)$. In addition, there are a lot of examples, Kneser graphs and alike, where the homological estimate for the chromatic number is tight.

One may even define the ``homological chromatic number'' of a graph $G$ as $\hind \mathcal B(H) + 2$, which serves as the lower bound for the honest chromatic number and for which Hedetniemi's conjecture holds true, in view of \eqref{equation:hind-prod}, which is explained in Section \ref{section:hom-chromatic}. On the way of developing this topological technique, we also prove the formula for the homological index of a join in Theorem \ref{theorem:join-index-2}, which may be of independent interest and which directly implies \eqref{equation:hind-prod}. 

We also establish a slightly weaker form of \eqref{equation:hind-prod} for the free action of prime cyclic groups $\mathbb Z/p$ (with odd prime $p$) on polyhedra, see Corollary \ref{corollary:product-index-p}. This leads to a topological lower bound on the categorical product of hypergraphs; introduced by Zhu \cite{zhu1992}.


\subsection{Acknowledgment}
The first author wishes to thank Professor Hossein Hajiabolhassan for drawing the author's attention to Hedetniemi's conjecture.

\section{Homological index of the join and the product}
\label{section:h-ind-2}

\subsection{Definition of the homological index}

Let us recall the standard definition. Speaking about topological spaces, let us restrict ourselves to polyhedra (simplicial complexes), which is sufficient in the application to graphs and their chromatic number. Throughout this paper, $G$ stands for a finite non-trivial group. Whenever such a group act on a polyhedron $X$, we assume there exists a finite triangulation of $X$ invariant with respect to the group action.

For $G$-spaces $X$ and $Y$, its Cartesian product $X\times Y$ is always considered as a $G$-space equipped with the diagonal action, i.e., $g\cdot (x,y)\mapsto (gx, gy)$. The join $X\ast Y$ of two topological spaces $X$ and $Y$ is a the quotient space $X\times Y\times [0,1]/\sim$ where $\sim$ is an equivalence relation generated by $(x, y_1, 0)\sim (x, y_2, 0)$ for all $x\in X$, and $y_1, y_2\in Y$ and $(x_1, y, 1)\sim (x_2, y, 1)$ for all $x_1, x_2\in X$ and $y\in Y$. For convenience, we use the notation $(1-t)x\oplus ty$ for $[x,y,t]\in X\ast Y$. Finally, if
$X$ and $Y$ are $G$-spaces, then $X\ast Y$ is considered as a $G$-space whose its action is given by $g\cdot \left((1-t)x\oplus ty\right)\mapsto (1-t)(gx)\oplus t(gy)$.

\begin{definition}
For a polyhedron $X$ with a free action of a finite group $G$ there exists a unique up to $G$-equivariant homotopy map $X \to EG$ to the classifying space of free $G$-actions. It gives rise to the cohomology map
\[
H_G^*(EG; M) = H^*(BG; M)\to H_G^*(X; M) = H^*(X/G; M),
\]
where $BG = EG/G$ and $M$ is a $G$-module.
\end{definition}

We refer the reader to the textbooks \cite{matousek2003using,hsiang1975} for more detailed information on these notions. Now we restrict our attention to the free action of the group $\mathbb Z/p$ of prime order.

\begin{definition}
For a polyhedron $X$ with a free action of $\mathbb Z/p$, the homological index of $X$ is
\[
\hind X = \max\{k\where H^k(BG;\mathbb F_p) \to H^k(X/G;\mathbb F_p)\quad\text{is nonzero}\},
\]
where $\mathbb F_p$ is the field of order $p$.
\end{definition}

The crucial property of the homological index is its monotonicity, the existence of an equivariant map $X\to Y$ implies $\hind X \le \hind Y$. It is clear from the definition.

\subsection{The Smith exact sequences and the transfer}

Let $X$ be a polyhedron (simplicial complex) with a free action of a cyclic group $G=\mathbb Z/p$. Denote by $(C_*(X;\mathbb F_p),\partial)$ the chain complex of oriented simplices of $X$ with coefficients in the field $\mathbb F_p$. Recall that $C_k(X;\mathbb F_p)$ consists of all sums $\sum n_i c_i$ where $n_i\in \mathbb F_p$, and $c_i$ are faces of $X$ of dimension $k$. Put 
\[
N:=\sum\limits_{g\in G} g = 1 + T + \dots + T^{p-1}\quad\text{and}\quad\rho:=1-T,
\]
where $T$ is a generator of the cyclic group $G$. It is known (see~\cite{bredon1972}) that $C_k(X/G;\mathbb F_p)$ can be identified with the space of $G$-invariant chains $NC_k(X;\mathbb F_p)$ and we have for $p=2$ (in this case $N=\rho=1+T$) an exact triple of complexes:
\[
0\to NC_*(X;\mathbb F_2)\to C_*(X;\mathbb F_2)\to NC_*(X;\mathbb F_2)\to 0
\]
and two exact triples for $p>2$:
\[
\begin{aligned}
0\to NC_*(X;\mathbb F_p)\to C_*(X;\mathbb F_p)\to \rho C_*(X;\mathbb F_p)\to 0; \\
0\to \rho C_*(X;\mathbb F_p)\to C_*(X;\mathbb F_p)\to NC_*(X;\mathbb F_p)\to 0
\end{aligned}
\]
These and induced long homology sequences, called \emph{the Smith sequences}, are functorial in $X$. For the equivariant homology we have: $H_*(NC_*(X;\mathbb F_2))=H_*(X/G;\mathbb F_p)$. Hence for $p=2$ we have a long exact sequence of homology groups:
\[
\dots\to H_n(X/G;\mathbb F_2)\to H_n(X;\mathbb F_2) \to H_n(X/G;\mathbb F_2) \to H_{n-1}(X/G;\mathbb F_2)\to \dots.
\]
For $p>2$ there are two long homology Smith sequences. The homology of the complex $\rho C_*(X;\mathbb F_p)$ is usually denoted by $H^{\rho}_*(X)$ and called \emph{the Smith special homology group}.

Note that any invariant (under the action of $G$) cycle belongs to $NC_*(X;\mathbb F_p)$ and therefore defines two homology classes, of the space $X/G$ and and of the space $X$. There arises the homomorphism $H_n(X/G;\mathbb F_p)\to H_n(X;\mathbb F_p)$, which is called \emph{the transfer homomorphism}.

\subsection{Homology of the join}

Now consider the join of two polyhedra $X$ and $Y$. Let $c=[a_0,a_1,\dots,a_n]$ and $d=[b_0,b_1,\dots,b_m]$ be faces in $X$ and $Y$ of dimensions $n$ and $m$ respectively, where $a_i$ are vertices of $X$ and $b_j$ are vertices of $Y$. Then the join of those faces $c*d=[a_0,a_1,\dots,a_n,b_0,b_1,\dots,b_m]$ is an $(n+m+1)$-dimensional face of $X*Y$ and we obviously have the boundary formula
\[
\partial (c*d)=(\partial c)*d+(-1)^{n+1}c*(\partial d).
\]
From this formula it follows that the join of cycles is a cycle and the join of a boundary with a cycle is a boundary, hence there is a natural homomorphism
\[
H_n(X;\mathbb F_p)\otimes H_m(Y;\mathbb F_p)\to H_{n+m+1}(X*Y;\mathbb F_p).
\]
Moreover, since we have the field as coefficients in homology, the reduced homology groups $\widetilde H_{*}(X*Y;\mathbb F_p)$ can be described due to the isomorphism~\cite{milnor1956}:
\[
\widetilde H_{s+1}(X*Y;\mathbb F_p)=\sum_{k+l=s} \widetilde H_k(X;\mathbb F_p)\otimes \widetilde H_l(Y;\mathbb F_p).
\]

Now assume that $\dim X=n$ and $\dim Y=m$, where $n$ and $m$ are positive. 

Then $\dim X*Y=n+m+1$ and 
\[
H_{n+m+1}(X*Y;\mathbb F_2)=H_n(X;\mathbb F_2)\otimes H_m(Y;\mathbb F_2).
\]
Note that in the top dimension the homology group can be identified with the group of cycles, so every nontrivial cycle defines a nonzero homological class. In the presence of an $G=\mathbb Z/p$ action, it follows that the inclusion $NC_n(X;\mathbb F_2)\to C_n(X;\mathbb F_2)$ defines a monomorphism 
\[
H_n(X/G;\mathbb F_2)\to H_n(X;\mathbb F_2).
\] 
Similarly, for spaces $Y$ and $X*Y$ in dimensions $m$ and $n+m+1$ there exists the similar monomorphisms. If we have nontrivial cycles $c\in C_n(X;\mathbb F_2)$ and $d\in C_m(Y;\mathbb F_2)$ then $c*d$ is a nontrivial $(n+m+1)$-dimensional cycle of $X*Y$, i.e. a nontrivial homology class. Moreover, if cycles $c$ and $d$ are invariant (under the $G$-action) then $c*d$ is also invariant.

\subsection{Subadditivity of the homological index}

We are going to use the property of the subadditivity of the homological index for $G=\mathbb Z/p$. Assume that a $G$-space $X$ is covered with its $G$-invariant subpolyhedra $X=A\cup B$. Choose a partition of unity $\phi+\psi = 1$ on $X$, which is $G$-invariant and subordinated to the considered covering. Then the map
\[
x\mapsto \phi(x)x\oplus \psi(x)x
\]
can be considered as a continuous $G$-equivariant map $X\to A*B$. Hence in order to estimate the index of $X$ from above using the indexes of $A$ and $B$, it is sufficient to estimate the index of $A*B$ from above. 

The needed information on subadditivity is collected in the following lemma, copied from \cite[Proposition 3.6]{volovikov2005}:

\begin{lemma}
\label{lemma:subadditive}
For a group $G=\mathbb Z/p$ and spaces it acts on, we have that
\[
\hind A\cup B \le \hind A*B \le \hind A + \hind B + 1,
\]
assuming $p=2$ or one of $\hind A$ and $\hind B$ is odd. Without the additional assumptions there holds 
\[
\hind A\cup B \le \hind A*B \le \hind A + \hind B + 2.
\]
\end{lemma}

\subsection{Spaces with $\mathbb Z/2$ action}

Now fix $G=\mathbb Z/2$ and establish the formula for the homological index of the join and the product of two $\mathbb Z/2$-spaces. 

\begin{theorem}
\label{theorem:join-index-2}
Let $X$ and $Y$ be polyhedra with free actions of $\mathbb Z/2$, then we have for their join
\[
\hind X*Y = \hind X + \hind Y + 1.
\]
\end{theorem}

\begin{proof}
The upper bound on $\hind X*Y$ is the subadditivity Lemma \ref{lemma:subadditive}, hence we concentrate on proving the lower bound. We first reduce the problem to the case $\dim X=n=\hind X$ and $\dim Y=m=\hind Y$. Indeed, taking the $n$-skeleton of $X$, we have that the map $H_n(X^{(n)}/G;\mathbb F_2)\to H_n(X/G;\mathbb F_2)$ is surjective and the adjoint cohomology map is injective. Hence the homological index $\hind X^{(n)}$ is not less than the index $\hind X$, but since it cannot be greater than the dimension, we have $\hind X^{(n)} = n$. The same can be done to $Y$ and the reduction is clear in view of the monotonicity $\hind X^{(n)}*Y^{(m)} \le \hind X*Y$.

Consider equivariant simplicial maps $X\to S^n$, $Y\to S^m$. Possibly after a subdivision of $X$, $Y$, $S^n$, and $S^m$, such maps exist from the dimension considerations in the obstruction theory.

Denote by $\alpha$ the sum of all $n$-dimensional simplicies of a $G$-invariant triangulation of $S^n$ and similarly denote be $\beta$ the sum of all $m$-dimensional simplicies of a $G$-invariant triangulation of $S^m$; that is the cycles representing the fundamental class of the manifolds $S^n$ and $S^m$. 

Then $\alpha$ and $\beta$ are invariant cycles and hence they give us generators of the groups 
\[
\begin{aligned}
H_n(S^n;\mathbb F_2)&=\mathbb F_2,\quad H_n(S^n/G;\mathbb F_2)=H_n(\mathbb RP^n;\mathbb F_2)=\mathbb F_2\,\, \mbox{ and }\\
H_m(S^m;\mathbb F_2)&=\mathbb F_2,\quad H_m(S^n/G;\mathbb F_2)=H_m(\mathbb RP^m;\mathbb F_2)=\mathbb F_2.
\end{aligned}
\] 
The invariant cycle $\alpha*\beta$ is a sum of all simplices for $S^{m+n+1}$ (here we do not use the orientation, but if we did, the orientations would be matching), hence $\alpha*\beta$ gives a generator of $H_{n+m+1}(S^{n+m+1};\mathbb F_2)=\mathbb F_2$.

Since $\dim X=n=\hind X$ and $\dim Y=m=\hind Y$, there exist an invariant $n$-dimensional cycle $\gamma$ of $X$ whose image in $S^n$ equals $\alpha$, and an invariant $m$-dimensional cycle $\delta$ of $Y$, whose image in $S^m$ equals $\alpha$. The join of the maps $X*Y\to S^n*S^m$ maps the invariant cycle $\gamma *\delta$ to $\alpha*\beta$. Hence the homology map 
\[
H_{n+m+1} (X*Y/G; \mathbb F_2) \to H_{n+m+1} (S^{n+m+1}/G; \mathbb F_2)
\]
is surjective, while the adjoint cohomology map is injective. This implies 
\[
\hind X*Y\ge \hind S^{n+m+1} = m+n+1,
\] 
and since $\dim X*Y=n+m+1$, we obtain $\hind X*Y= m+n+1$.
\end{proof}

\begin{corollary}
\label{corollary:prod-index-2}
Let $X$ and $Y$ be polyhedra with free actions of $\mathbb Z/2$, then we have for their product
\[
\hind X\times Y = \min\{\hind X, \hind Y\}.
\]
\end{corollary}
\begin{proof}
Considering the points of the join $X*Y$ in the form $(1-t)x\oplus ty$, we cover the joint with two open sets: one with $t\neq 1/2$, and the other with $t\in(1/3, 2/3)$. The first set equivariantly deforms to the disjoint union $X\sqcup Y$ with 
\[
\hind X\sqcup Y = \max\{\hind X, \hind Y\}
\]
right from the definition of the homological index. The second set equivariantly deforms to $X\times Y$. By the subadditivity of the homological index we have
\[
\hind X*Y \le \hind X\sqcup Y + \hind X\times Y + 1.
\]
Inserting the lower bound for $\hind X*Y$, we then obtain
\[
\hind X\times Y \ge \hind X + \hind Y - \max\{\hind X, \hind Y\} = \min\{\hind X, \hind Y\}.
\]
The reverse inequality evidently holds from the existence of equivariant maps $X\times Y\to X$ and $X\times Y\to Y$.
\end{proof}

\subsection{Spaces with $\mathbb Z/p$ action}

The methods of the previous section, in a certain part, can be extended to the case of the group $G=\mathbb Z/p$ with odd $p$.

\begin{theorem}
\label{theorem:join-index-p}
Let $X$ and $Y$ be polyhedra with free actions of $\mathbb Z/p$, when $p$ is an odd prime. If $\hind X$ and $\hind Y$ are odd then for the index of the join we have
\[
\hind X*Y = \hind X + \hind Y + 1.
\]
\end{theorem}

\begin{proof}
The upper bound on $\hind X*Y$ is the subadditivity Lemma \ref{lemma:subadditive}, and the proof of the lower bound from Theorem \ref{theorem:join-index-2} in fact works, since the odd-dimensional spheres $S^n$ and $S^m$ can be considered unit spheres of a complex vector space, and the action of $G$ can be introduced as the multiplication by $e^{i\frac{2\pi k}{p}}$ for $k = 0,\ldots, p-1$.

The considerations of the obstruction theory give $G$-equivariant maps $X\to S^n$, $Y\to S^m$, and the homological index of any of the spheres $S^n$, $S^m$, $S^{n+m+1}$ equals its dimension because of its connectivity. 
\end{proof}

\begin{corollary}
Let $X$ and $Y$ be polyhedra with free actions of $\mathbb Z/p$, then for their product there holds
\[
\hind X\times Y = \min\{\hind X, \hind Y\},
\]
provided that $\hind X$ and $\hind Y$ are both odd.
\end{corollary}
\begin{proof}
We cover the join $X*Y$ with two sets, one equivariantly deformable to $X\sqcup Y$ and the other equivariantly deformable to $X\times Y$. We have
\[
\hind X\sqcup Y = \max\{\hind X, \hind Y\}
\]
from the definition of the homological index, and in our case this is an odd number. By Lemma~\ref{lemma:subadditive} we have
\[
\hind X*Y \le \hind X\sqcup Y + \hind X\times Y + 1.
\]
Inserting the lower bound for $\hind X*Y$, we then obtain
\[
\hind X\times Y \ge \hind X + \hind Y - \max\{\hind X, \hind Y\} = \min\{\hind X, \hind Y\}.
\]
The reverse inequality evidently holds from the existence of equivariant maps $X\times Y\to X$ and $X\times Y\to Y$.
\end{proof}

\begin{corollary}
\label{corollary:product-index-p}
Let $X$ and $Y$ be polyhedra with free actions of $\mathbb Z/p$, then without any assumption on the indexes for their product there holds
\[
\min\{\hind X, \hind Y\} - 1 \le \hind X\times Y \le \min\{\hind X, \hind Y\}
\]
and the upper bound is attained when $\min\{\hind X, \hind Y\}$ is odd.
\end{corollary}
\begin{proof}
The upper bound follows from the existence of equivariant maps $X\times Y\to X$ and $X\times Y\to Y$, thus we discuss the lower bound.

Assume $n = \hind X \le m = \hind Y$ without loss of generality. We may replace $Y$ with its skeleton of homological index and dimension precisely $n$, as in the proof of Theorem \ref{theorem:join-index-p} and pass to the case $\hind X = \hind Y = n$. If $n$ is odd then we are done by the previous corollary, because the upper bound is attained. Otherwise we, again, pass to spaces with index smaller by one and apply the case of odd index from the previous corollary.
\end{proof}

\section{Mapping index}
\label{section:mapping-index}

\subsection{Definition and the index of the product conjecture}

Let us now turn to the mapping index of a polyhedron $X$ with a free action of a finite group $G$.

\begin{definition}
\label{definition:mapping-index}
For a polyhedron $X$ with a free action of a finite group $G$, the \emph{mapping index} $\ind X$ is the minimal $k$ such that there exists a $G$-equivariant map
\[
X \to E_kG = \underbrace{G*G*\dots *G}_{k+1}
\]
to the $k$-dimensional approximation of $EG$. Similarly the mapping co-index $\coind X$ is the maximal $k$ such that there exists a $G$-equivariant map
\[
E_kG \to X.
\]
\end{definition}

In the case $G=\mathbb Z/2$ the space $B_kG$ is topologically a sphere $S^k$ with the antipodal action of $G$, given by $x\mapsto -x$; hence the definition is about equivariant maps to spheres and from spheres. The mapping index also has the monotonicity property, the existence of an equivariant map $X\to Y$ implies $\ind X \le \ind Y$. This is clear from the definition.

Since the dimension of $E_kG$ and $B_kG = E_kG/G$ is $k$, its homological index is at most $k$ and from the monotonicity of the homological index one readily obtains the inequality $\hind X \le \ind X$, hence the homological index is a lower bound for the mapping index.

In \cite{wrochna2019} it was shown that Hedetniemi's conjecture for graphs would imply the equality
\begin{equation}
\label{equation:ind-prod}
\ind X\times Y = \min\{\ind X, \ind Y\},
\end{equation}
for polyhedra $X$ and $Y$ with a free action of $G=\mathbb Z/2$. In view of the failure of Hedetniemi's conjecture the status of \eqref{equation:ind-prod} remains open; we will discuss its validity for free actions of arbitrary finite group $G$. 

The simplest observation is that there exists $G$-equivariant maps $X\times Y\to X$ and $X\times Y\to Y$. Using the monotonicity of the mapping index we readily obtain
\[
\ind X\times Y \le \min\{\ind X, \ind Y\},
\]
and therefore the whole question is if this inequality is always sharp or not.

We also mention the \emph{subadditivity property} of the mapping index of the join
\[
\ind X*Y \le \ind X + \ind Y + 1.
\]
It is easier compared to the subadditivity of the homological index in Lemma~\ref{lemma:subadditive}, since from equivariant maps $X\to E_nG$ and $Y\to E_mG$ we directly obtain, by taking their join, an equivariant map
\[
X*Y \to E_nG * E_m G = E_{n+m+1} G.
\]

\subsection{Some partial results on the mapping index of the product}

Here we present some of our findings about the conjectural equality of the mapping index of a product.

\begin{lemma}
\label{Lem:1}
If $X$ and $Y$ are $G$-spaces then there is an equivariant map $$(G* X)\times Y\rightarrow G*(X\times Y).$$
\end{lemma}
\begin{proof}
It is easy to check that the following map
\[
\varphi :(G* X)\times Y \longrightarrow G*(X\times Y),\quad ((1-t)g\oplus t x, y) \longmapsto (1-t)g \oplus t (x,y) 
\]
defines an equivariant map from $(G*X)\times Y$ to $G*(X\times Y)$.
\end{proof}

\begin{lemma}
\label{lemma:with-a-map}
If there exists an equivariant map $X\to Y$ or $Y\to X$ then $\ind X\times Y = \min\{\ind X, \ind Y\}$.
\end{lemma}
\begin{proof}
Assume there exists equivariant $\phi : X\to Y$, then $\id_{X}\times\phi$ gives an equivariant map from $X$ to $X\times Y$ where $\id_{X}$ is the identity map on $X$. Hence $\ind X\leq \ind X\times Y$ from the monotonicity and in total 
\[
\min\{\ind X, \ind Y\}\le \ind X \le \ind X\times Y
\]
implies the result, since $\min\{\ind X, \ind Y\}\ge \ind X\times Y$ always holds.
\end{proof}

\begin{theorem}
\label{Thm:23}
Let $X$ be a $G$-space with $\ind X < \infty$. Then $$\ind X\times E_nG=\min\{\ind X,\, n\}.$$
\end{theorem}
\begin{proof}
We need to consider the following cases:

1) Assume $\ind X\leq n$. Then by the definition of the mapping index there exists an equivariant map $h:X\to E_nG$ and we apply Lemma~\ref{lemma:with-a-map}. 

2) Assume $m=\ind X > n$. By using the first case, we have 
$\ind\left(E_mG\times X\right)=m$. 
Note that $E_nG*\left(\underbrace{G*\cdots* G}_{m-n}\right)= E_{m}G$. Thus,
by Lemma \ref{Lem:1}, there exists an equivariant map
\[
E_mG\times X\longrightarrow \left(\underbrace{G*\cdots* G}_{m-n}\right)*(E_nG\times X).
\]
Therefore,
\[
m=\ind E_mG\times X\leq \ind \left(\underbrace{G*\cdots* G}_{m-n}\right)*(E_nG\times X)\leq (m-n)+  \ind E_nG\times X,
\]
where the last inequality follows from the subadditivity of the mapping index of joins. This implies that $n\leq  \ind E_nG\times X$, which completes the proof.
\end{proof}

In the following, we show the correctness of the equality \eqref{equation:ind-prod}  for several cases. In particular, we construct several examples in order to answer a question raised by Marcin Wrochna~\cite{wrochna2019} as follows: ``Can one compute the index of some non-trivial products involving non-tidy spaces?''. Let us first clarify what a tidy space mean:

\begin{definition}
A $G$-space $X$ is called a \emph{tidy space} if $\ind X = \coind X$. 
\end{definition}

Note that the case that both of $X$ and $Y$ are tidy spaces is a special case of Lemma~\ref{lemma:with-a-map}. Indeed, if $m=\coind X=\coind X\leq \coind Y= \ind Y=n$, then there is an equivariant map from $X$ to $E_m G$ as $m=\ind X$, and an equivariant map from $E_m G$ to $Y$ as $m\leq n=\coind Y$. The composition of these two maps defines an equivariant map from $X$ to $Y$. In summary, for a non-trivial example we need at least one of the factor in the product be a non-tidy space. The next theorem establishes some particular cases of the calculation of the index of the product of $G$-spaces:

\begin{theorem}
\label{Thm:25}
Let $X, Y$ be polyhedra with free actions of a finite group $G$. If
\begin{enumerate}[label=(\alph*)]
\item at least one of $X$ and $Y$ is a tidy space, or
\item $G=\mathbb Z/2$, and $\ind X=\hind X\leq \hind Y\leq\ind Y$, or
\item $G=\mathbb Z/2$, and $X$, $Y$ are closed topological manifolds with $\ind X= \dim  X$ and $\ind Y= \dim  Y$, then
\end{enumerate}
$$\ind  X\times Y= \min\{\ind X,  \ind  Y\}.$$
\end{theorem}
\begin{proof}
Suppose we have the assumption in part $(a)$. Without loss of generality we can assume that $X$ is a tidy space. Put $m= \coind X= \ind X$. There is an equivariant map $f: E_mG\to X$ as $\coind X= m$. Also, the identity map $\id_{Y}$ is an equivariant map form $Y$ to itself. Now, the map $f\times \id_Y$ defines an equivariant map from $E_{m}G\times Y$ to $X\times Y$. In the similar way there is an equivariant map from $X\times Y$ to $E_{m}G\times Y$ as $\ind X= m$. Therefore, $\ind X\times Y= \ind E_{m}G\times Y$. Now using Theorem~\ref{Thm:23} we get the desired conclusion. 

For the part $(b)$, the following inequalities show the correctness of our assertion.
\begin{multline*}
\min\{\ind  X, \ind  Y\} = \min\{\hind X, \hind Y\} = \\
= \hind X\times Y \leq \ind X\times Y \leq \min\{\ind  X, \ind  Y\},
\end{multline*}
where the second equality follows from Corollary~\ref{corollary:prod-index-2}. 

For the proof of the last part it is sufficient to note that, by \cite[Corollary 3.1]{musin2015borsuk}, the assumption $\ind M=\dim M$ implies $\hind  M=\ind M=\dim M$ when $M$ is a closed topological manifold (or a pseudomanifold modulo $2$). Therefore, 
\[
\hind X\times Y=\min \{\hind X,\hind Y\}=\min \{\dim  X,\dim Y\}.
\]
Since the homological index is a lower bound for the mapping index and 
\[
\ind X\times Y\leq \min \{\dim X,\dim Y\},
\]
we are done. Please note that results of this kind hold true for more general spaces than pseudomanifolds modulo $2$, see Proposition~3.1 and Corollaries~3.1 and 3.2 in \cite{musin2015borsuk}.
\end{proof}

\subsection{Particular examples of the mapping index of the product}

Let us list some concrete non-trivial examples satisfying \eqref{equation:ind-prod}. First note that any pair of $G$-space $X$, $Y$ with the following properties
\begin{equation}
\label{equation:index-interlaced}
\coind Y < \coind X\leq \ind X < \ind Y
\end{equation}
is a non-trivial candidate for Wrochna's question. Indeed, there is no equivariant map from $X$ to $Y$ as $\coind X > \coind Y$. Similarly, there is no equivariant map from $Y$ to $X$ as $\ind Y > \ind X$.

\begin{example}
Consider $X=S^{n}$ as a free $\mathbb Z/2$-space whose action is given by the antipodal map, i.e., $x\mapsto -x$ for all $x\in S^{n}$. The famous Borsuk--Ulam theorem says there is no $\mathbb Z/2$-map from a higher dimensional sphere to a lower dimensional sphere, this is the reason for the definition of the mapping index. Therefore, $\coind X=\ind X= n $. Also, consider the projective space $Y=\mathbb{R}P^{2n+1}$ as the quotient of the unit sphere $S^{2n+1}$ in $\mathbb{C}^{n+1}$ over the relation $x\sim -x$ for all $x\in S^{2n+1}$. The map $(v_1, \ldots, v_{n+1})\mapsto (iv_1, \ldots, iv_{n+1})$ where $v_1, \ldots, v_{n+1}\in \mathbb{C}$, induces a free $\mathbb  Z/2$-action on $\mathbb{R}P^{2n+1}$. It is known that $\coind Y\leq 1$~\cite{zivaljevic2002level}, and $\ind Y\geq n+1$~\cite{stolz1989level}. Hence, for each $n\geq 2$, the pair $X$, and $Y$ satisfies \eqref{equation:index-interlaced} and the assumption $(a)$ of Theorem~\ref{Thm:25}. Therefore, this pair is a non-trivial example of free $\mathbb Z/2$-spaces satisfying \eqref{equation:ind-prod}.
\end{example}

In the previous example, one of the factors is a tidy space. Next example shows the computation of index of the product of two non-tidy spaces.

\begin{example}
First let us start by construction a non-tidy space $X$ satisfying the following condition:
\[
2\leq \coind X < \hind X= \ind X.
\]
In~\cite[Theorem 6.6]{conner1960fixed} Conner and Floyd assert that once we have a vector bundle $\beta : E \to Z$ with dimension of the fibers $n$ and the degree of the largest nonzero dual Stiefel--Whitney class $k$ 
(meaning $\bar w_k(\beta)\neq 0$) then its bundle of $(n-1)$-spheres  $S(\beta)$ has homological-index precisely $n-1+k$. Note that Conner and Floyd use a different notation interchanging index and co-index, while
Clapp and Marzantowicz \cite{clapp2000essential} called the index ``capacity". In our terminology $\mathrm{co-ind}_2$ of Conner and Floyd means the
homological index.

It is easy to find an example of such $\mathbb Z/2$-space when $\hind$ and $\ind$ (in our 
notation) coincide. Let $\gamma_{k,n} \to G_{k,n}$ be the tautological bundle over the Grassmannian. Then $\bar w_{n-k}(\gamma_{k,n})  = w_{n-k}(\gamma_{n-k,n})\neq 0$ and hence by the above theorem $\hind (S(\gamma_{k,n})) = n-1$. But a member of $S(\gamma_{k,n})$ is just a unit vector in $\mathbb R^n$ together with a $k$-dimensional linear subspace containing it. Hence, forgetting the linear subspace, we easily map $S(\gamma_{k,n})$ to $S^{n-1}$ equivariantly with respect to taking the opposite. This proves that $\hind$ and $\ind$ of this space equal $n-1$. Now, let $n\geq 3$ be odd and $X= S(\gamma_{3,n})$, the existence of an equivariant map $f : S^{n-1}\to S(\gamma_{3,n})$ would mean that the composition with the equivariant map $g : S(\gamma_{3,n}) \to S^{n-1}$,
\[
S^{n-1}\longrightarrow S(\gamma_{3,n}) \longrightarrow S^{n-1},
\]
has odd degree. We will only need that this degree is nonzero and the Euler class of the vector bundle $\alpha = (g\circ f)^*(TS^{n-1})$ is therefore nonzero. The existence of $f$ means that $\alpha$ has a $2$-dimensional subbundle $\delta$, formed by the orthogonal complement of the unit vector in its attached (in $S(\gamma_{3,n})$) $3$-dimensional linear subspace of $\mathbb R^n$. But the Stiefel--Whitney and Chern classes of $\delta$ are zero just 
because of the dimension considerations (the cohomology of $S^{n-1}$ vanishes in degrees 1 and 2); and for a two-dimensional vector bundle $\delta$ this is sufficient to show that the bundle is trivial. Hence $\delta$ has a nonzero section, and $\alpha$ also has a nonzero section. But this contradicts the fact that the Euler class of $\alpha$ is nonzero. The contradiction confirms that 
\[
2\leq\coind S(\gamma_{3,n}) < n-1.
\]

On the other hand, in~\cite{mmm} Matsushita constructed a free $\mathbb Z/2$-manifold $X_n$ for each integer $n\geq 2$ such that $\hind X_{n}= \ind X_n= n$, but $\coind X_{n}\leq 1$. Put $Y= X_{n}$. 
Now, by \eqref{equation:index-interlaced} and the assumption $(b)$ in Theorem~\ref{Thm:25}, the pair $X$ and $Y$ is another non-trivial example of free $\mathbb Z/2$-spaces satisfying \eqref{equation:ind-prod}, both of them being non-tidy.
\end{example}

\section{Box complex and the homological chromatic number}
\label{section:hom-chromatic}
In order to bound the chromatic number of a given graph $H$ from below, the box complex $\mathcal B(H)$ has been introduced in \cite{matousek2003using}. Actually, $\mathcal B(H)$ is an abstract simplicial complex whose simplices are all
$A_1\times\{1\}\cup A_2\times\{2\}$ with $A_1, A_2$ are disjoint subsets of $V(H)$, and every vertex of $A_1$ is connected to every vertex of $A_2$ in $H$. Moreover if one of the part ($A_1$ or $A_2$) is empty, then the other part must have a common neighbor in $H$. There is a natural $\mathbb Z/2$-action on $\mathcal B(H)$ given by $(v,1)\mapsto (v,2)$ and $(v,2)\mapsto (v,1)$ for all $v\in V(H)$. This action turns the polyhedron of the geometric realization of $\mathcal B(H)$, $||\mathcal B(H)||$, into a free $\mathbb Z/2$-space. The following theorem reveals the main motivation behind the definition of this object.

\begin{theorem}[Theorem~5.9.3 in \cite{matousek2003using}]
\label{theorem:box}
For every finite graph $H$, $\chi(H)\geq \ind ||B(H)||+2$.
\end{theorem}

This theorem inspires the following definition:

\begin{definition}
For a graph $H$, call
\[
\hchi (H) = \hind \mathcal ||B(H)|| + 2
\]
the \emph{homological chromatic number} of $H$. 
\end{definition}

By Theorem \ref{theorem:box}, in view of the fact that the homological index estimates the mapping index from below, the homological chromatic number gives a lower bound for the ordinary chromatic number of a graph, that is 
\[
\chi (H) \ge \hchi (H)
\]
always. Now we are in position to mention our promised ``true version'' of Hedetniemi's conjecture for homological chromatic number. First, recall that the categorical product $H_1\times H_2$ of two graphs $H_1$ and $H_2$, is the graph with the vertex set $V(H_1)\times V(H_2)$ and two vertices $(u_1,u_2)$ and $(v_1,v_2)$ are adjacent in $H_1\times H_2$ if and only if $u_i$ and $v_i$ are adjacent in $H_i$ for $i=1, 2$.

\begin{corollary}[Hedetniemi's conjecture for homological chromatic number]
\label{coro: h-ind}
For every graphs $H_1$ and $H_2$ we have
$$\hchi (H_1 \times H_2)= \min\{\hchi (H_1), \hchi (H_2)\}.$$
\end{corollary}
\begin{proof}
It is known that for every pair of graphs $H_1$ and $H_2$, the box complex $||B(H_1\times H_2)||$ is $\mathbb Z/2$-equivariantly homotopy equivalent to $||\mathcal B(H_1)||\times ||\mathcal B(H_2)||$, see~\cite[Remark 3]{simonyi2010topological}. Therefore
\[
\hind ||B(H_1\times H_2)||=\hind ||\mathcal B(H_1)||\times ||\mathcal B(H_2)||.
\] 
Now, Corollary~\ref{corollary:prod-index-2} gives the desired result.
\end{proof}

The next corollary  automatically implies the validity of Hedetniemi's conjecture for graphs, whose homological chromatic number equals the actual chromatic number.

\begin{corollary}
For every graphs $H_1$ and $H_2$, we have
\[
\chi(H_1\times H_2)\geq\min\{\hchi (H_1), \hchi (H_2)\}.
\]
\end{corollary}
\begin{proof}
\[
\chi(H_1\times H_2) \geq  \hchi (H_1\times H_2) = \min\{\hchi (H_1), \hchi (H_2)\}
\]
by Corollary~\ref{coro: h-ind}.
\end{proof}

It is worth mentioning that the above corollary was known, if we replace $\hind$ by $\coind$ in the definition of homological chromatic number, see~\cite{simonyi2010topological}. But our bound can significantly better than the bound given by co-index, as:
\begin{itemize}
    \item For any free $\mathbb Z/2$-space $X$, $\coind X\leq \hind X$. 
    \item Any free $\mathbb Z/2$-simplicial complex is $\mathbb Z/2$-homotopy equivalent to $B(H)$ for some graph $H$, see~\cite{csorba2007homotopy}.
    \item The difference between the homological index and the co-index can be arbitrary large~\cite{mmm}.
\end{itemize}
\section{Categorical products of hypergraphs}
The aim of this section is to present a topological lower bound for the categorical products of hypergraphs. Before proceed, let us first review some basic facts on hypergraphs. 
\\~\\
\textbf{Hypergraphs:}
A hypergraph $\mathcal{H}$ is a pair $\mathcal{H} = (V, E)$ where $V$ is a finite set of elements called vertices, and $E$ is a set of non-empty subsets of $V$ called edges. 
The vertex set and the edge set of a hypergraph $\mathcal{H}$ are often denoted by $V(\mathcal{H})$ and $E(\mathcal{H})$, respectively. An $m$-coloring of a hypergraph $\mathcal{H}$ is a map $c: V(\mathcal{H})\to\{1,\ldots, m\}$. Moreover, $c$ is called a proper $m$-coloring if it creates no monochromatic edge, i.e., $|c(e)|\geq 2$ for all $e\in\mathcal{H}$. We say a hypergraph is $m$-colorable if it admits a proper $m$-coloring. The chromatic number of a hypergraph $\mathcal{H}$, denoted by $\chi (\mathcal{H})$, is the minimum $m$ such that $\mathcal{H}$ is $m$-colorable. For two hypergraphs $\mathcal{H}_1 = (V_1, E_1)$ and $\mathcal{H}_2 = (V_2, E_2)$ a homomorphism from $\mathcal{H}_1$ into $\mathcal{H}_2$ is a mapping $\psi : V_1\to V_2$ which sends any edge of $\mathcal{H}_1$ to an edge in $\mathcal{H}_2$, i.e.,  $\psi (e)\in E_2$ for all $e\in E_1$. An $r$-uniform hypergraph is a hypergraph such that all its edges have size $r$. A $2$-uniform hypergraph is simply called a graph. 
\\~\\

In $1992$, a generalization of Hedetniemi's conjecture was raised by Xuding Zhu~\cite{zhu1992chromatic}. Actually, Zhu conjectured that for given hypergraphs $\mathcal{H}_1$ and $\mathcal{H}_2$ we have 
$$\chi (\mathcal{H}_1\times\mathcal{H}_2)=\min\{\chi(\mathcal{H}_1), \chi(\mathcal{H}_2)\},$$
where the categorical product $\mathcal{H}_1\times\mathcal{H}_2$ is the hypergraph with the vertex set
$V(\mathcal{H}_1)\times V(\mathcal{H}_2)$
and whose edge set is
$$E(\mathcal{H}_1\times\mathcal{H}_2)=\{\{(u_1, v_1),\ldots, (u_s, v_s)\} : s\geq 1, \{u_1, \ldots, u_s\}\in E(\mathcal{H}_1), \{v_1, \ldots, v_s\}\in E(\mathcal{H}_2)\}.$$
Note that the product of two graphs with the above definition is not the same as the categorical product of two graphs mentioned earlier. Indeed, if $H_1$ and $H_2$ are two graphs and each has at least one edge, then $H_1\times H_2$ with this new product is not even a graph, since it has edges of size $3$ and $4$. However, it is fairly easy to see that the graph categorical product of two graphs and the hypergraph product of the same two graphs have the same chromatic number. This confirms the claim that Zhu's conjecture is a generalization of Hedetniemi's conjecture. Although Shitov's result~\cite{shitov2019} shows this conjecture is not true in general, we believe it is an interesting question to find families of hypergraphs which satisfy the conjecture. The first nontrivial family of hypergraphs satisfying this conjecture was presented in $2016$ by H. Hajiabolhassan and F. Meunier~\cite{hajiabolhassan2016hedetniemi}. Actually, they proved that Zhu's conjecture is true for every pair of usual Kneser $r$-uniform hypergraphs. For more recent work, see~\cite{sani2018coloring}.

In this section, we define the compatibility $r$-uniform hypergraph $\mathcal{C}_P^{(r)}$ which is assigned to a $G$-poset $P$; a partially ordered set equipped with a group action $G$. Then, we establish a connection between the chromatic number of this hypergraph and a topological property of the order complex of $P$. Via this connection, we will present a topological lower bound for the chromatic number of the categorical product of two hypergraphs. Consequently, we enrich the family of known hypergraphs satisfying the conjecture. In particular, we will obtain the Hajiabolhassan--Meunier result via this bound. To mention our results precisely, we need some definitions that will be presented in the next subsection. 

\subsection{Compatibility $r$-uniform hypergraph}
In this subsection, we define the compatibility $r$-uniform hypergraph, the main object of our study here. First, we review some standard facts on $G$-posets.
\\~\\
\textbf{G-posets:}
Here and subsequently, $G$ stands for a finite non-trivial group and its identity element is denoted by $e$. Remember, any two elements $a, b$ of a poset $(P,\preceq)$ are said comparable if $a\preceq b$ or $b\preceq a$. 
\begin{definition}
A $G$-poset is a poset $(P, \preceq)$ equipped with a $G$-action on its elements that preserves the partial order, i.e., $a\preceq b$ implies $g\cdot a\preceq g\cdot b$ for all $g\in G$ and $a, b\in P$.
\end{definition}
A $G$-poset $P$ is called free $G$-poset, if for all $p\in P$, $g\cdot p=p$ implies $g=e$.  Recall, the order complex $\Delta (P)$ of a poset $(P, \preceq)$ is the simplicial complex whose vertices are the elements of $P$ and whose simplices are all chains in $P$; i.e, $$\Delta (P)=\{\{x_1\prec\cdots\prec x_n\}: n\geq 1\quad x_1,\ldots, x_n\in P\}.$$ Finally, a face poset $P(\mathcal{K})$ of a simplicial complex $\mathcal{K}$ is a poset whose elements are simplices of $K$ ordered by inclusion. In addition, if $\mathcal{K}$ is a (free) $G$-simplicial complex, then its face poset $P(\mathcal{K})$ is a (free) $G$-poset with the induced action of $G$. 
\\~\\
\begin{definition}[compatibility $r$-uniform hypergraph]
Let $P$ be a $G$-poset, and $r\geq 2$ be a positive integer. The compatibility $r$-uniform hypergraph of $P$, denoted by ${\mathcal{C}}_P^{(r)}$, 
is the $r$-uniform hypergraph that has $P$ as vertex set and whose subset $e=\{p_1,\ldots, p_r\}$ of r$ $distinct elements of $P$ forms an edge of ${\mathcal{C}}_P^{(r)}$ if and only if for each distinct elements $p_i, p_j\in e$ there is an element $g\in G\setminus\{e\}$ such that $p_i$ and $g\cdot p_j$ are comparable in $P$.
\end{definition}
It is worth pointing out that the graph version of compatibility $r$-uniform hypergraph, i.e., when $r=2$ in Definition $1$, is defined in~\cite{daneshpajouh2017topological}. Throughout the paper, we simply write $\mathcal{C}_{P}$ instead of $\mathcal{C}_{P}^{(2)}$. We refer the interested reader to~\cite{daneshpajouh2018dold, daneshpajouh2018new, daneshpajouh2017topological} for more information on compatibility graphs and their applications. We emphasize that some variants of compatibility graphs, for $G=\mathbb Z/2$, were defined before by several authors~\cite{csorba2007homotopy, vzivaljevic2005wi, walker1983graphs}. Now, we are in a position to state our first result.
\begin{theorem}
\label{Thm:2}
Let $p$ be a prime number, $\mathbb Z/p$ the cyclic group of order $p$, and $P$ a free $\mathbb Z/p$-poset. Then, for every positive integer $r$ with $2\leq r\leq p$ we have:
$$\left\lceil\frac{\ind\left(||\Delta (P)||\right)+1}{r-1}\right\rceil+ \left\lceil\frac{p}{r-1}\right\rceil-1\leq\chi\left(\mathcal{C}_P^{(r)}\right).$$
\end{theorem}
\begin{proof}
For simplicity of notation, set $G=\mathbb Z/p$, $n= \chi\left(\mathcal{C}_P^{(r)}\right)$, and $m= n-\left\lceil\frac{|G|}{r-1}\right\rceil+1$. Let $\sigma_{G}^{(r-2)}$ be a simplicial complex whose vertex set is $G$ and whose simplices are all subsets of $G$ of size at most $(r-1)$. Note that $G$ acts freely on $\sigma_{G}^{(r-2)}$. Now, suppose that $c : V\left(\mathcal{C}_P^{(r)}\right)\to\{1, \ldots , n\}$ is a proper coloring of $\mathcal{C}_P^{(r)}$. We show that this coloring induces a simplicial $G$-equivariant map
\begin{align*}
  \lambda: \Delta (P)\longrightarrow \underbrace{\sigma_{G}^{(r-2)}\ast\cdots\ast\sigma_{G}^{(r-2)}}_{m}.
\end{align*}
Note that such a map on the level of vertices must take $P$ to $G\times\{1,\ldots, m\}$.

We partition $P$ into disjoint equivalence classes, the orbits, under the $G$-action. Remember, the equivalence class containing $x$ is $[x]=\{g\cdot x: g\in G\}$. Take an arbitrary class $[x]$. We pick one element, say $x^{\prime}$, from $[x]$ with the property that
$$c(x^{\prime})= \min\{c(y) : y\in [x]\},$$
as a representative, and we set $\lambda (x^{\prime}) = (e, c(x^{\prime}))$. Then, we extend $\lambda$ on the remaining elements of $[x]$ as follows.
$$\lambda(g\cdot x^{\prime})= (g,c(x^{\prime}))\quad\forall g\in G.$$
We repeat the same procedure for the other classes as well.

Let us verify that $\lambda$ is a well-defined function. For this purpose, we need to show that any point in $[x^{\prime}]$, say $x$, can be uniquely represented as $x= g\cdot x^{\prime}$ for some $g\in G$. Now, suppose that $x= g\cdot x^{\prime}= h\cdot x^{\prime}$ for some $g, h\in G$. Then $(h^{-1}g)\cdot x^{\prime}= x^{\prime}$. So $h^{-1}g=e$, as $P$ is a free $G$-poset. Thus, $h=g$. Moreover, for every $x\in P$ we have $|\{c(y): y\in [x]\}|\geq \left\lceil\frac{|G|}{r-1}\right\rceil$, as every $r$-subset of $[x]$ forms an edge in ${C}_P^{(r)}$. Therefore, $\lambda$ takes its values in $G\times\{1,\ldots, m\}$.

Next, we show that $\lambda$ is a simplicial $G$-equivariant map. Clearly, by the definition of $\lambda$, this map preserves the $G$-action. So, to prove our claim we just need to show that $\lambda$ is a simplicial map, i.e, takes any simplex to a simplex. Note that $F_1\uplus\cdots\uplus F_{m}$ is a simplex of $\underbrace{\sigma_{G}^{(r-2)}\ast\cdots\ast\sigma_{G}^{(r-2)}}_{m}$ if and only if $|F_i|\leq (r-1)$ for $i= 1, \ldots, m$. Therefore, $\lambda$ is a simplicial map if for every  $x_1\prec\cdots\prec x_r\in P$  with $\lambda(x_1)=(a_1, b_1),\ldots, \lambda(x_r)=(a_r, b_r)$, if all $b_i$'s are equal, then all $a_i$'s are not pairwise disjoint. Now, let $x_1\prec\cdots\prec x_r$ be $r$ distinct elements of $P$ such that:
$$\left(\text{for}\quad i=1,\ldots,r\quad\lambda (x_i) = (g_i, c(x_i^{\prime}))\right)\quad \&\quad \left(c(x_1^{\prime})=\ldots=c(x_r^{\prime})\right),$$
where $x_i^{\prime}$ is the representative of the class $[x_i]$. To finish the proof, we need to show that all $g_i'$s cannot be pairwise distinct. Since otherwise, on the one hand we have $x_i^{\prime} = g_i^{-1}\cdot x_i$ for all $1\leq i\leq r$. On the other hand, for every $1\leq i < j\leq r$ 
$$\underbrace{g_i^{-1}\cdot x_i}_{x_i^{\prime}}\prec g_i^{-1}\cdot x_j = (g_i^{-1}g_j)\underbrace{g_j^{-1}\cdot x_j}_{x_j^{\prime}}.$$ 
These imply that $\{x_1^{\prime}, \ldots x_r^{\prime}\}$ is an edge in $\mathcal{C}_P^{(r)}$, as $g_i^{-1}g_j\neq e$ for all $i\neq j$. This contradicts the fact that $c$ is a proper coloring of $\mathcal{C}_P^{(r)}$. Therefore, $\lambda$ is a $G$-simplicial map. This map naturally induces a $G$-equivariant map from $||\Delta (P)||$ to $||\underbrace{\sigma_{G}^{(r-2)}\ast\cdots\ast\sigma_{G}^{(r-2)}}_{m}||$.  Therefore, according to monotonicity of mapping-index and the fact that the mapping index of a free $G$-polyhedra is bounded above by its dimension \cite[Proposition 6.2.4]{matousek2003using}, we have
\begin{align*}
  \ind ||\Delta (P)||\leq\dim\left(||\underbrace{\sigma_{G}^{(r-2)}\ast\cdots\ast\sigma_{G}^{(r-2)}}_{m}||\right) & = m(r-1)-1\\
  & = (\chi(\mathcal{C}_P^{(r)})-\left\lceil\frac{p}{r-1}\right\rceil+1)(r-1)-1,
\end{align*}
as $||\underbrace{\sigma_{G}^{(r-2)}\ast\cdots\ast\sigma_{G}^{(r-2)}}_{m}||$ is free.
Consequently,
$$\left\lceil\frac{\ind ||\Delta (P)||+1}{r-1}\right\rceil+ \left\lceil\frac{p}{r-1}\right\rceil-1\leq\chi\left(\mathcal{C}_P^{(r)}\right),$$ 
which is the desired conclusion. 
\end{proof}
To solve a conjecture of Erd\H{o}s~\cite{erdos1973problems}, a generalization of the box complex for $r$-uniform hypergraphs was introduced by Alon--Frankl--Lov\'{a}sz~\cite{alon1986chromatic} as follows.
\begin{definition}[$B_{\text{edge}}(\mathcal{H})$]
Let $\mathcal{H}$ be an $r$-uniform hypergraph,   
\begin{align*}
  \pi_i :  \underbrace{V(\mathcal{H})\times\cdots\times V(\mathcal{H})}_{r} & \longrightarrow V(\mathcal{H})\\
  & (u_1,\ldots, u_i, \ldots , u_r)\longmapsto u_i,
\end{align*}
the projection on the $i$-th coordinate, and $\pi_{i}(T)=\{\pi_i(x): x\in T\}$ for any $T\subseteq{\left(V(\mathcal{H})\right)}^r$. The Alon--Frankl--Lov\'{a}sz complex of $\mathcal{H}$ is the simplicial complex $B_{\text{edge}}(\mathcal{H})$\footnote{They originally used the notation $C(\mathcal{H})$ instead of $B_{\text{edge}}(\mathcal{H})$. We have this notation from~\cite{thansri2012simple}.} whose vertices are all the ordered $r$-tuples $(u_1,\ldots, u_r)$ of
vertices of $\mathcal{H}$ with the property that $\{u_1,\cdots , u_r\}\in E(\mathcal{H})$. A set $T=\{(u^i_l,\ldots,u^i_r): i\in I\}$ of vertices of $B_{\text{edge}}(\mathcal{H})$
forms a simplex if $\pi_1(T),\ldots,$ $\pi_r(T)$ are pairwise disjoint and $\{x_1, \cdots, x_r\}\in E(\mathcal{H})$ for any choosing $x_1\in\pi_1(T), \ldots, x_r\in\pi_r(T)$.
\end{definition}
Note that the cyclic group $\mathbb Z/r$ acts on the poset $B_{\text{edge}}(\mathcal{H})$ naturally by cyclic shift, and turns it to a free $\mathbb Z/r$-simplicial complex. Here and subsequently, we consider $B_{\text{edge}}(\mathcal{H})$ as a $\mathbb Z/r$-simplicial complex with the mentioned $\mathbb Z/r$-action. Similar to the graph version, the following theorem reveals the main motivation behind the definition of this object.
\begin{theorem}[The Alon--Frankl--Lov\'{a}sz bound~\cite{alon1986chromatic}\footnote{The lower bound was given in term of a weaker topological parameter, connectivity, in the original paper.}]
\label{Thm: alon}
Let $p$ be a prime number and $\mathcal{H}$ be an $p$-hypergraph. Then   
$$\chi\left(\mathcal{H}\right)\geq 1+ \left\lceil\frac{\ind||B_{\text{edge}}(\mathcal{H})||+1}{p-1}\right\rceil.$$
\end{theorem}
As a simple application of Theorem \ref{Thm:2}, let us recover the Alon--Frankl--Lov\'{a}sz bound here. Consider the face poset of $B_{\text{edge}}(\mathcal{H})$, $P(B_{\text{edge}}(\mathcal{H}))$, as a free $\mathbb Z/r$-poset with the induced action of $B_{\text{edge}}(\mathcal{H})$. 
\begin{lemma}
\label{lem: 9}
Let $r\geq 2$ be a positive integer, and $\mathcal{H}$ be an $r$-uniform hypergraph. Then, there is a hypergraph homomorphism from $\mathcal{C}_{P(B_{\text{edge}}(\mathcal{H}))}^{(r)}$ to $\mathcal{H}$ which in particular implies that $\chi(\mathcal{C}_{P(B_{\text{edge}}(\mathcal{H}))}^{(r)})\leq\chi(\mathcal{H})$.  
\end{lemma}
\begin{proof}
During the proof, the symbol $|X|$ is used for the cardinality of the set $X$. Also, let $\zeta$ be the generator of the cyclic group $\mathbb Z/r$, i.e, $\mathbb Z/r=\{e=\zeta^0, \zeta^1,\ldots, \zeta^{r-1}\}$. Now, we show that any map like $\psi : \mathcal{C}_{P(B_{\text{edge}}(\mathcal{H}))}^{(r)}\to \mathcal{H}$ that sends each element $T=\{(u^i_1,\ldots,u^i_r): i\in I\}$ of $\mathcal{C}_{P(B_{\text{edge}}(\mathcal{H}))}^{(r)}$ to an arbitrary element of $\pi_{1}(T)$ is a hypergraph homomorphism. Suppose that the vertices 
$$T_1=\{(u^i_1,\ldots,u^i_r): i\in I_1\}, \ldots , T_r=\{(u^i_1,\ldots,u^i_r): i\in I_r\}$$ forms an edge in $\mathcal{C}_{P(B_{\text{edge}}(\mathcal{H}))}^{(r)}$. To verify our claim, we need to show that $\{u_1, \ldots, u_r\}$ is an edge of $\mathcal{H}$ when $u_1\in \pi_1(T_1), \ldots, u_r\in\pi_r(T_r)$. Without loss of generality assume $|T_1|\leq\cdots\leq |T_r|.$
Now, by definition of compatibility hypergraph, for any $1\leq j\leq r-1$ there is an $1\leq s_j\leq r-1$ such that $T_{j}\subseteq\zeta^{s_j}\cdot T_{r}$. These imply that
 \[                                                                                                                               
\left\{                                                                                                                             \begin{array}{rcl}                                                                                                           
u_1\in \pi_1(T_1) & \subseteq & \pi_{(s_{1}+1\mod r)}(T_r)\\                                                                  
\cdots  & & \cdots \\                                                                                                          
u_{r-1}\in \pi_1(T_{r-1}) & \subseteq & \pi_{(s_{r-1}+1\mod r)}(T_r)\\                                                         u_{r}\in \pi_1(T_r)                                                                                                           
\end{array}                                                                                                                     
\right.                                                                                                                              \]
We claim that all the $s_i$'s are pairwise distinct. Suppose the contrary, that is there exist $1\leq i < j\leq r-1$ such that 
$s_i= s_j$. Set $s_i= s_j= a$. Also, there is a $1\leq b\leq r-1$ such that $T_i\subseteq\zeta^b\cdot T_j$.
Therefore
\[   
     \begin{cases}
        T_i\subseteq\zeta^{a}\cdot T_{r}\\
        T_j\subseteq\zeta^{a}\cdot T_{r}\\
        T_i\subseteq\zeta^{b}\cdot T_j\
     \end{cases}
\Longrightarrow
     \begin{cases}
       T_i\subseteq\zeta^{a} \cdot T_{r}\\
        T_i\subseteq\zeta^{a+b}\cdot T_{r}\
     \end{cases}
\Longrightarrow 
     \begin{cases}
       \pi_1(T_i)\subseteq \pi_{(a+1 \mod r)}(T_{r})\\
       \pi_1(T_i)\subseteq \pi_{(a+b+1 \mod r)}(T_{r})\
     \end{cases}
\]
Therefore, $\pi_{(a+1\mod r}(T_{r})\cap \pi_{(a+b+1\mod r)}T_{r}\neq\emptyset$. This contradicts the fact that $\pi_{l}(T_r)\cap \pi_{s}(T_r)=\emptyset$ for any two distinct $l, s$. Thus, all $u_i$'s are pairwise distinct, i.e., $|\psi(T)|= r$. Moreover, they forms an edge in $\mathcal{H}$, as no pair of distinct $u_i$'s belongs to a same $\pi_l(T_r)$ for some $1\leq l\leq r$. This finishes the proof. 
\end{proof}
Now, for a prime number $p$ and $p$-hypergraph $\mathcal{H}$, by Lemma~\ref{lem: 9} and Theorem~\ref{Thm:2}, we have
\begin{align*}
    \chi(\mathcal{H})\geq \chi(\mathcal{C}_{P(B_{\text{edge}}(\mathcal{H}))}^{(p)}) & \geq\\
    & \left\lceil\frac{\ind\left(||\Delta (P(B_{\text{edge}}(\mathcal{H})))||\right)+1}{p-1}\right\rceil+ \left\lceil\frac{p}{p-1}\right\rceil-1=\\
    & \left\lceil\frac{\ind||B_{\text{edge}}(\mathcal{H})||+1}{p-1}\right\rceil+1,
\end{align*}
where the last equality comes from that fact that any $G$-simplicial 
complex and its barycentric subdivision are $G$-homeomorphic. 
\subsection{A topological lower bound for the chromatic number of the categorical product of hypergraphs}
In this final subsection, we present our aforementioned topological lower bound for the chromatic number of the categorical product of hypergraphs.
\begin{corollary}
\label{Thm:3}
Let $p$ be a prime number, $\mathbb Z/p$ the cyclic group of order $p$, and $P, Q$ free $\mathbb Z/p$-posets. Then, for every positive integer $r$ with $2\leq r\leq p$ we have:
    $$\left\lceil\frac{\min\{\coind ||\Delta (P)||, \coind ||\Delta (Q)||\}+1}{r-1}\right\rceil+ \left\lceil\frac{p}{r-1}\right\rceil-1\leq\chi\left({\mathcal{C}}_P^{(r)}\times\mathcal{C}_Q^{(r)}\right).$$
\end{corollary}
The proof is based on two lemmas. First, let us recall the definition of a product of two posets. Let $(P_1, \preceq_1)$, $(P_2, \preceq_2)$ be two posets. The product $P_1\times P_2$ becomes a poset with the following relation; $(p_1, p_2)\preceq (q_1, q_2)$ if $p_1\preceq_1 q_1$ and $p_2\preceq_2 q_2$. Furthermore, if they are (free) $G$-posets, then $P_1\times P_2$ is (free) $G$-poset with the natural action, i.e, $g\cdot (p_1, p_2)= (g\cdot p_1, g\cdot p_2)$.  
\begin{lemma}
\label{lem:7}
If $P$ and $Q$ are free $G$-posets, then $\mathcal{C}_{P\times Q}^{(r)}$ is a sub-hypergraph of $\mathcal{C}_P^{(r)}\times \mathcal{C}_Q^{(r)}$. In particular,
$\chi (\mathcal{C}_{P\times Q}^{(r)})\leq\chi (\mathcal{C}_P^{(r)}\times \mathcal{C}_Q^{(r)}).$
\end{lemma}

\begin{proof}
First note that both of these hypergraphs have the same vertex set, $P\times Q$. Now, let $A=\{(p_1,q_1),\ldots, (p_r,q_r)\}$ be an edge of $\mathcal{C}_{P\times Q}^{(r)}$. We need to verify that $A$ is also an edge of $\mathcal{C}_P^{(r)}\times \mathcal{C}_Q^{(r)}$. This is equivalent to showing that $A_1=\{p_1, \ldots, p_r\}\in E(\mathcal{C}_P^{(r)})$ and $A_2=\{q_1, \ldots, q_r\}\in E(\mathcal{C}_Q^{(r)})$. Take two arbitrary elements $p_i, p_j\in\ A_1$. There is a $g\in G\setminus\{e\}$ such that $(p_i, q_i)$, and $(g\cdot p_j, g\cdot q_j)$ are comparable in $P\times Q$. Thus, $p_i$ and $g\cdot p_j$ are comparable in $P$. Also $p_i\neq p_j$, as $P$ is a free $G$-poset which implies $|A_1|=r$. Therefore, $A_1\in E(\mathcal{C}_P^{(r)})$. Similarly, $A_2\in E(\mathcal{C}_Q^{(r)})$, and the proof is completed. 
\end{proof}
\begin{lemma}[\cite{walker1988canonical}]
\label{lem:8}
If $P$ and $Q$ are free $G$-posets, then 
$$||\Delta(P\times Q)||\cong_{G}||\Delta(P)||\times ||\Delta(Q)||.$$
In other words, there is a homeomorphism between $||\Delta(P\times Q)||$ and $||\Delta(P)||\times ||\Delta(Q)||$ which preserves the $G$-action. 
\end{lemma}
Now, the proof of Corollary~\ref{Thm:3} is easily deduced from Theorem~\ref{Thm:2}, Lemmas \ref{lem:7}, \ref{lem:8}, and the following simple fact:
\begin{itemize}
\item  If $X, Y$ are $G$-spaces, then $\coind X\times Y =\min\{\coind X, \coind Y\}$.
\end{itemize}
Now, we present a topological lower bound for the chromatic number of the Categorical product of arbitrary $r$-uniform hypergraphs, provided $r$ is a prime number.
\begin{corollary}
\label{cor: 10}
Let $p$ be a prime number. For $p$-uniform hypergraphs $\mathcal{H}_1, \mathcal{H}_2$ we have
\begin{align*}
\chi\left(\mathcal{H}_1\times \mathcal{H}_2\right) \geq
 1+ \left\lceil\frac{\min\{\coind ||B_{\text{edge}}(\mathcal{H}_1)||, \coind ||B_{\text{edge}}(\mathcal{H}_2)||\}+1}{p-1}\right\rceil.
\end{align*}
\end{corollary}
\begin{proof}
\begin{multline*}
    \chi(\mathcal{H}_1\times \mathcal{H}_2)\geq 
     \chi(\mathcal{C}_{P(B_{\text{edge}}(\mathcal{H}_1))}^{(p)}\times \mathcal{C}_{P(B_{\text{edge}}(\mathcal{H}_2))}^{(p)})\geq\\
     1+ \left\lceil\frac{\min\{\coind ||B_{\text{edge}}(\mathcal{H}_1)||), \coind ||B_{\text{edge}}(\mathcal{H}_2)||\}+1}{p-1}\right\rceil,
\end{multline*}
where the first inequality comes from Lemma~\ref{lem: 9} and the second one is deduced from Corollary~\ref{Thm:3}.
\end{proof}
As immediate consequences of Corollary~\ref{cor: 10} the following family are a new type of hypergraphs satisfying Zhu's conjecture.  
\begin{itemize}
    \item For a prime number $p$, any $p$-uniform hypergraphs $\mathcal{H}_1, \mathcal{H}_2$ with
    $$\chi\left(\mathcal{H}_i\right)= 1+ \left\lceil\frac{\coind ||B_{\text{edge}}(\mathcal{H}_i)||+1}{p-1}\right\rceil\quad\text{for $i=1, 2$}.$$
\end{itemize}
In particular, any pair of the usual Kneser $r$-uniform hypergraphs satisfy the previous condition satisfy 
Zhu's conjecture when $r$ is a prime number, see~\cite{alon1986chromatic}. Note that, one can reduce a non-prime case to a prime case with a combinatorial argument, see~\cite{Meunier, hajiabolhassan2016hedetniemi}. Thus, one may reprove the H. Hajiabolhassan and F. Meunier~\cite{hajiabolhassan2016hedetniemi} result (the usual Kneser $r$-uniform hypergraphs satisfying Zhu's conjecture) from this viewpoint.

Unfortunately, we could not replace $\coind$ with $\hind$ in Corollary~\ref{cor: 10} in general as we were not able to present the topological analogues of Hedetniemi's conjecture in full generality for homological-index of $\mathbb Z/p$-spaces when $p\geq 3$ is a prime number. However, we can have that by assuming some extra conditions.
\begin{corollary}
\label{cor: 11}
Let $p$ be a prime number. If for $p$-uniform hypergraphs $\mathcal{H}_1, \mathcal{H}_2$ we have
\begin{enumerate}
    \item $\min\{\hind ||B_{\text{edge}}(\mathcal{H}_1)||, \hind ||B_{\text{edge}}(\mathcal{H}_2)||\}$ is odd or 
    \item $\min\{\hind ||B_{\text{edge}}(\mathcal{H}_1)||, \hind ||B_{\text{edge}}(\mathcal{H}_2)||\}$ is not divisible by $p-1$,
    then
\end{enumerate}
\begin{align*}
\chi\left(\mathcal{H}_1\times \mathcal{H}_2\right) \geq
 1+ \left\lceil\frac{\min\{\hind ||B_{\text{edge}}(\mathcal{H}_1)||, \hind ||B_{\text{edge}}(\mathcal{H}_2)||\}+1}{p-1}\right\rceil.
\end{align*}
\end{corollary}
\begin{proof}
\begin{multline*}
\chi\left(\mathcal{H}_1\times \mathcal{H}_2\right)\geq \chi(\mathcal{C}_{P(B_{\text{edge}}(\mathcal{H}_1))}^{(p)}\times \mathcal{C}_{P(B_{\text{edge}}(\mathcal{H}_2))}^{(p)})\quad \text{(by Lemma~\ref{lem: 9})}\\
\geq \chi(\mathcal{C}_{{P(B_{\text{edge}}(\mathcal{H}_1))}\times P(B_{\text{edge}}(\mathcal{H}_2))}^{(p)})\quad
\text{(by Lemma~\ref{lem:7})}\\                                                 \geq 1+ \left\lceil\frac{\ind ||\Delta\left(P(B_{\text{edge}}(\mathcal{H}_1))\times P(B_{\text{edge}}(\mathcal{H}_2))\right)||+1}{p-1}\right\rceil \quad \text{(by Theorem~\ref{Thm:2})}\\                      
\geq 1+ \left\lceil\frac{\hind ||\Delta\left(P(B_{\text{edge}}(\mathcal{H}_1))\times P(B_{\text{edge}}(\mathcal{H}_2))\right)||+1}{p-1}\right\rceil \quad \text{(as $\ind X\geq \hind X$)} \\                   
= 1+ \left\lceil\frac{\hind ||\Delta\left(P(B_{\text{edge}}(\mathcal{H}_1))||\right)\times ||\Delta\left(P(B_{\text{edge}}(\mathcal{\
H}_2))\right)||+1}{p-1}\right\rceil \quad \text{(by Lemma~\ref{lem:8})}\\
= 1+ \left\lceil\frac{\hind ||B_{\text{edge}}(\mathcal{H}_1)||\times ||B_{\text{edge}}(\mathcal{H}_2)||+1}{p-1}\right\rceil\\    
= 1+ \left\lceil\frac{\min\{\hind ||B_{\text{edge}}(\mathcal{H}_1)||, \hind||B_{\text{edge}}(\mathcal{H}_2)||\}+1}{p-1}\right\rceil,
\end{multline*}
where the last equality comes directly form Corollary~\ref{corollary:product-index-p} if Condition (1) is assumed. If the second condition is assumed, then the last equality again is deduced from the lower bound presented in Corollary~\ref{corollary:product-index-p} and the fact that 
$$\left\lceil\frac{a+1}{p-1}\right\rceil= \left\lceil\frac{a}{p-1}\right\rceil,$$
where $a= \{\hind ||B_{\text{edge}}(\mathcal{H}_1)||, \hind||B_{\text{edge}}(\mathcal{H}_2)||\}$ provided that $a$ is not divisible by $p-1$.
\end{proof}

%


\bibliography{main.bib}
\bibliographystyle{abbrv}
\end{document}